\documentclass[final]{dmtcs-episciences}
\usepackage[utf8]{inputenc}
\usepackage{subfigure}
\usepackage{amsthm}
\usepackage{paralist}
\usepackage{color}
\usepackage{amsmath}
\usepackage{microtype}

\author{Michael W. Schroeder\affiliationmark{1}
  \and Rebecca Smith\affiliationmark{2}\thanks{partially supported by the NSA Young Investigator Grant H98230-08-1-0100}}
\title
{A Bijection on Classes Enumerated by the Schr\"oder Numbers}
\affiliation{
  Marshall University, Huntington, WV USA\\
  SUNY Brockport, Brockport, NY USA}
\received{2015-11-3}

\revised{2016-9-19, 2017-3-30, 2017-5-26}

\accepted{2017-6-7}

\newcommand{\doublehat}[1]{\begin{array}{@{}c@{\hspace{-.5pt}}}\\[-3.25ex]\hat{}\\[-3ex]\hat{}\\[-2.3ex]#1\\[-.6ex]\end{array}}
\newcommand{\doubleExphat}[1]{\scalebox{.8}{$\doublehat{#1}$}}

\newtheorem{theorem}{Theorem}[section]

\newtheorem{lemma}[theorem]{Lemma}
\newtheorem{corollary}[theorem]{Corollary}

\newtheorem{algorithm}[theorem]{Algorithm}

\theoremstyle{definition}
\newtheorem{example}[theorem]{Example}
\newtheorem{observation}[theorem]{Observation}
\newtheorem{definition}[theorem]{Definition}

\newcommand{\DI}[1]{\texttt{\tt #1}}
\newcommand{\pic}[2][1]{\includegraphics[scale=#1,page=#2]{schroder_paths.pdf}}
\renewcommand{\th}{\textsuperscript{th}}
\newcommand{\lew}[1][W]{<_{\scalebox{.5}{$#1$}}}
\newcommand{\precw}[1][W]{\prec_{\scalebox{.5}{$#1$}}}
\newcommand{\leqw}[1][W]{\leq_{\scalebox{.5}{$#1$}}}

\begin{document}
\publicationdetails{18}{2017}{2}{15}{1326}
\maketitle

\begin{abstract}
We consider a sorting machine consisting of two stacks in series where the first stack has the added restriction that entries in the stack must be in decreasing order from top to bottom.  
The class of permutations sortable by this machine is known to be enumerated by the Schr\"oder numbers.  
In this paper, we give a bijection between these sortable permutations of length $n$ and Schr\"oder paths of order $n-1$: the lattice paths from $(0,0)$ to $(n-1,n-1)$ composed of East steps $(1,0)$, North steps $(0,1)$, and Diagonal steps $(1,1)$ that travel weakly below the line $y=x$.
\end{abstract}

\section{Introduction}

A stack is a sorting device that works by a sequence of push and pop operations.  This last-in, first-out machine was shown by Knuth~\cite{knuth:the-art-of-comp:1} to sort a permutation if and only if that permutation avoids the pattern $231$.  That is, if there are not three indices $i<j<k$ with $\pi_k<\pi_i<\pi_j$, then it is possible to run $\pi$ through a stack and output the identity permutation.  The class of stack-sortable permutations is enumerated by the Catalan numbers.

In the language of permutation patterns, any downset of permutations in the permutation containment ordering is a \emph{class}, and every class has a \emph{basis}, which consists of the minimal permutations not in the class.  
Given that the basis for the class of permutations sortable by one stack contains only a single pattern of length three, considering two stacks in series is quite natural.  However, the problem becomes rather unwieldy.  In the case of two stacks in series, Murphy~\cite{murphy:restricted-perm:} showed that the class of sortable permutations has an infinite basis.  The enumeration of this class also appears to be difficult.  The best known bounds are given by Albert, Atkinson, and Linton~\cite{albert:permutations-ge:}.

We note that to sort a permutation by two stacks in series, the push and pop operations are such that when an entry is popped out of the first stack, it is immediately pushed into the second stack.

To get a better handle on this problem, many have considered different types of weaker sorting machines.  One such weaker machine is a stack in which the entries must increase when read from top to bottom.  
Atkinson, Murphy, and Ru\v{s}kuc~\cite{atkinson:sorting-with-tw:} found an optimal algorithm for sorting permutations which are sortable using two increasing stacks in series.  
Note that to obtain the identity permutation, the last stack will be an increasing stack even without declaring this restriction.
Their left-greedy algorithm sorts all permutations sortable by this machine.  
Interestingly enough, the basis for these sortable permutations is still infinite, but the permutation class was found to be in bijection with the permutations that avoid $1342$ as enumerated by B\'ona~\cite{bona:exact-enumerati:}.  Both enumerations were found by using a bijection with $\beta(0,1)$ trees.  

One can analogously define a \emph{decreasing stack} as a stack in which the entries must decrease when read from top to bottom.  Smith~\cite{smith:a-decreasing-st:} considered sorting with a decreasing stack followed by an increasing stack, a machine called DI (we refer to the decreasing stack as D and the increasing stack as I).  We illustrate how a permutation can be sorted with the DI sorting machine in Figure~\ref{DI_example}.  Notice that this permutation contains the pattern $2341$, and as such cannot be sorted by two increasing stacks in series.  
The class of DI-sortable permutations was shown to have a finite basis, $\{3142, 3241\}$.  Kremer~\cite{kremer:permutations-wi:,kremer:postscript:-per:} has shown previously that this class is enumerated by the large Schr\"oder numbers.  
\begin{figure}[ht]
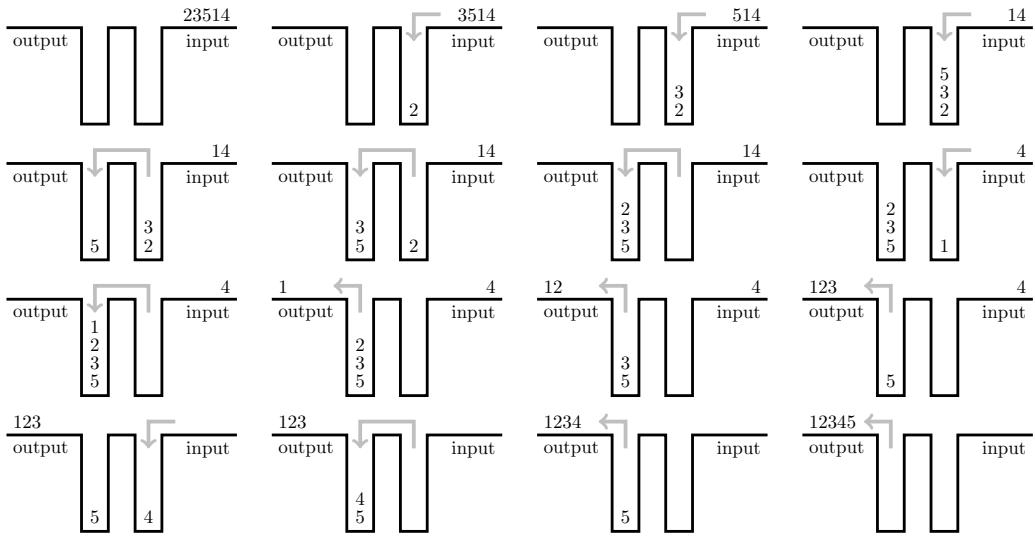

\centering
\scalebox{.95}{\begin{tabular}{cccc}
\pic[.75]{14} &
\pic[.75]{15} &
\pic[.75]{16} &
\pic[.75]{17} \\
\pic[.75]{18} &
\pic[.75]{19} &
\pic[.75]{20} &
\pic[.75]{21} \\
\pic[.75]{22} &
\pic[.75]{23} &
\pic[.75]{24} &
\pic[.75]{25} \\
\pic[.75]{26} &
\pic[.75]{27} &
\pic[.75]{28} &
\pic[.75]{29} \\
\end{tabular}}
\caption{Sorting the permutation $23514$.}
\label{DI_example}
\end{figure}

The Schr\"oder numbers were introduced in somewhat modern times by Schr\"oder~\cite{schroder:vier} as legal bracketing of variables, though at least the initial terms of this sequence were known to Hipparchus~\cite{stanley:schroder}.  
Rogers and Shapiro~\cite{rogers:schroder} found bijections showing certain classes of lattice paths are enumerated by the Schr\"oder numbers.
While there are several such classes of lattice paths, we use the one given in the definition below.

\begin{definition}
Let $n\geq 1$.
A lattice path from $(0,0)$ to $(n,n)$  taking only East $(1,0)$, North $(0,1)$, and Diagonal $(1,1)$ steps while staying weakly below the main diagonal will be referred to as a \emph{Schr\"oder path}.  
We illustrate all of the Schr\"oder paths from $(0,0)$ to $(2,2)$ in Figure~\ref{paths}.
\end{definition}

\begin{figure}[ht]
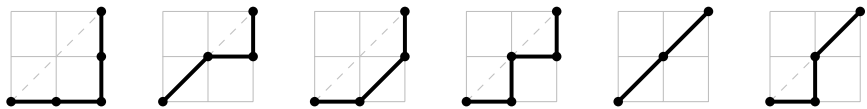

\centering
\begin{tabular}{cccccc}
 \pic{8} &
 \pic{9} &
\pic{10} &
\pic{11} &
\pic{12} &
\pic{13} \\
\end{tabular}
\caption{Schr\"oder paths from $(0,0)$ to $(2,2)$.}
\label{paths}
\end{figure}

In Section \ref{sec:sorting_algorithm}, we state the algorithm from \cite{smith:a-decreasing-st:} for producing a word from a DI-sortable permutation, then classify when a sorting word is the result of the algorithm.
In Section \ref{sec:bijection}, we give an algorithm which produces a Schr\"oder path from a word produced by the algorithm in Section \ref{sec:sorting_algorithm}, then show there is a bijective correspondence between algorithmic DI words and Schr\"oder paths.
In Section \ref{sec:properties}, we discuss some properties of DI-sortable permutations which can be ascertained from aspects of their corresponding Schr\"oder paths, and we conclude with some open questions.

On a related note, Ferrari~\cite{ferrari:permutation-classes} gave bijections between permutations sorted by restricted deques and a different class of lattice paths also enumerated by the Schr\"oder numbers.  
Also, Bandlow, Egge, and Killpatrick~\cite{bandlow:schroder} gave a bijection between a different permutation class and Schr\"oder paths.  Additionally, more background on stacks in general was consolidated by Bona~\cite{bona:a-survey-of-sta:}.

\section{The Sorting Algorithm}
\label{sec:sorting_algorithm}
Let $W$ be a word, $i$ be a positive integer, and \DI{L} be a letter.
In what follows, we denote the letter in the $i$\textsuperscript{th} position of $W$ as $W(i)$, 
$\#_\DI{L}(W)$ as the number of occurrences of \DI{L} in $W$, and 
$W^i$ as the first $i$ characters in $W$.
We often refer to $W^i$ as a \emph{prefix} of $W$.
We use $\DI{L}^i$ to denote $i$ repetitions of \DI{L}.

In \cite{smith:a-decreasing-st:}, Smith showed the following algorithm was an optimal way to sort permutations using the DI machine in the sense that this algorithm would sort any DI-sortable permutation.

\begin{algorithm}~\label{sorting_algorithm}
\begin{enumerate}
\item{If the top entry of the second (increasing) stack  is the next entry of the output, then pop the entry to the output.}
\item{If all of the $m$ entries in the first (decreasing) stack  make up the next $m$ entries of the output, then push those entries to the second stack.}
\item{Otherwise, if the next entry of the input is smaller than the top entry of the second stack and larger than the top entry of the first stack, then push it onto the first stack.  
We will apply the convention that the next entry of the input satisfies each of the aforementioned properties if there is no entry in the corresponding stack with which to compare it.  
Thus, this step can be thought of as pushing the next entry from the input to the top of the first stack if that entry can legally sit atop each of the two stacks at this stage.}
\item{Finally if neither of those moves are available, push the top entry from the first stack to the second.}
\end{enumerate}
\end{algorithm}

It was noted in \cite{smith:a-decreasing-st:} that Step 2 of this algorithm is not needed to produce an optimal algorithm.  
However, we continue to utilize it as outputting entries earlier in the process is useful in constructing a bijection between the DI-sortable permutations and Schr\"oder paths.

\begin{definition} 
\label{def:enc}
Let $\DI{E}$ represent a push step from the input to the first stack, let $\DI{N}$ represent pop/push from the first stack to the second stack, and let $\DI{C}$ represent the pop from the second stack to the output.
See Figure~$\ref{fig:defenc}$.
A \emph{sorting word} of a permutation is a word representing steps that can be taken to sort that permutation using two stacks in series (without restriction).
We use $\DI{L}_i$ to denote the occurrence of the letter \DI{L} in $W$ which corresponds to the movement of symbol $i$.
\end{definition}

\begin{figure}[h]
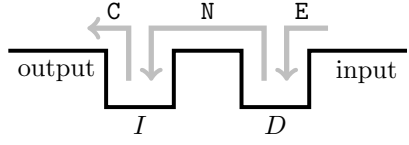

\centering
\pic{43}
\caption{The operations \DI{E}, \DI{N}, and \DI{C} corresponding to the DI machine from Definition \ref{def:enc}.}
\label{fig:defenc}
\end{figure}
\begin{observation}~\label{observation}
A word $W$ is a sorting word of a sortable permutation of length $n$ if and only if the following two conditions are met:
\begin{enumerate}
\item The length of $W$ is $3n$ and $W$ contains exactly $n$ of each of the letters $\DI{E}$, $\DI{N}$, and $\DI{C}$.
\item For all $x\in[3n]$, $\#_\DI{E}(W^x) \geq \#_\DI{N}(W^x) \geq \#_\DI{C}(W^x)$.
That is, in any prefix of $W$, the number of \DI{C}s does not exceed the number of \DI{N}s, which does not exceed the number of \DI{E}s.
\end{enumerate}
Note that for every sorting word, there exists a permutation $\pi$ which it sorts.
To find $\pi$, apply $W$ to the identity permutation, and let $\sigma$ be the output (which is not necessarily sorted).
Then $W$ sorts $\pi=\sigma^{-1}$. 
\end{observation}

\begin{definition} 
A \emph{DI word} of a permutation $\pi$ is a sorting word representing steps that can be taken to sort $\pi$ using the DI machine.  
\end{definition}

\begin{example}  
The permutation $23514$ has four DI words: \DI{EEENNNENCCCENCC} (illustrated in Figure \ref{DI_example}), \DI{EEENNNENCCECNCC}, \DI{EEENNNENCECCNCC}, and \DI{EEENNNENECCCNCC}.
\end{example}

\begin{definition}  
The \emph{algorithmic DI word} (ADI word for short) of a permutation $\pi$ is the unique word representing steps that will be taken when applying Algorithm~\ref{sorting_algorithm} to sort $\pi$ by the DI machine.
\end{definition}

\begin{example}  
\label{ex:adi}
Algorithm \ref{sorting_algorithm} was used in the process illustrated by Figure \ref{DI_example}.
So the ADI word of $23514$ is \DI{EEENNNENCCCENCC}.
To use the indices outlined in Definition~\ref{def:enc}, the ADI word of $23514$ can also be given as 
$\DI{E}_2
\DI{E}_3
\DI{E}_5
\DI{N}_5
\DI{N}_3
\DI{N}_2
\DI{E}_1
\DI{N}_1
\DI{C}_1
\DI{C}_2
\DI{C}_3
\DI{E}_4
\DI{N}_4
\DI{C}_4
\DI{C}_5
$.
\end{example}

For the duration of the paper, we say that $\DI{X} \lew \DI{Y}$ in a word $W$ if $\DI{X}$ appears before $\DI{Y}$ in $W$.
Similarly we say $\DI{X}\precw\DI{Y}$ if $\DI{X}$ appears immediately before $\DI{Y}$ in $W$.

We begin with some observations about the relative locations of symbols in a DI-word, then follow with a classification of when a sorting word is a DI-word and an ADI-word.

\begin{observation}
\label{obs:wordprops}
Let $\pi$ be a DI-sortable permutation and let $W$ be a DI word.
Let $i,j\in[n]$.
\begin{enumerate}[(i)]
\item $\DI{E}_{\pi_i}\lew\DI{E}_{\pi_j}$ if and only if $i < j$; symbols appearing in $\pi$ earlier enter the first stack earlier.
\item\label{obs:wordpropsCC} $\DI{C}_{\pi_i}\lew\DI{C}_{\pi_j}$ if and only if $\pi_i < \pi_j$; symbols must exit the second stack in increasing order.
\item\label{obs:wordpropsCN} 
If $\pi_i < \pi_j$ and $\DI{N}_{\pi_i}\lew\DI{N}_{\pi_j}$, then $\DI{C}_{\pi_i}\lew\DI{N}_{\pi_j}$; 
if a smaller number enters the second stack before a larger number, it must be output before the larger number can enter the second stack -- otherwise the increasing condition for the second stack is not satisfied.
Hence if $\pi_i < \pi_j$ and $\DI{N}_{\pi_i}\lew\DI{E}_{\pi_j}$, then $\DI{C}_{\pi_i}\lew\DI{N}_{\pi_j}$.
We highlight this specifically as it will be used in later proofs.
\item\label{obs:wordpropsNNCCC} If $\DI{N}_{\pi_i} \lew \DI{N}_{\pi_j} \lew \DI{C}_{\pi_i}$, then $\DI{C}_{\pi_j}\lew\DI{C}_{\pi_i}$;
if $\pi_j$ enters the second stack after $\pi_i$ and before $\pi_i$ exits, then $\pi_j < \pi_i$ by the increasing condition of the second stack. So $\pi_j$ must exit before $\pi_i$.
\item\label{obs:wordpropsNE2} If $\DI{E}_{\pi_j} \lew \DI{C}_{\pi_i} \lew \DI{N}_{\pi_j}$ (which implies $\pi_i < \pi_j$), then $\DI{N}_{\pi_i} \lew \DI{E}_{\pi_j}$;
if a smaller number is pushed to the output while a larger number is in the first stack, the smaller number was popped from the first stack before the larger number entered the first stack -- otherwise the larger number must be pushed to the second stack before the smaller number is pushed to the second stack.
\item\label{obs:wordpropsENE} If $\pi_i < \pi_j$ and $\DI{E}_{\pi_j}\lew\DI{E}_{\pi_i}$, then $\DI{E}_{\pi_j} \lew \DI{N}_{\pi_j} \lew\DI{E}_{\pi_i}$;
if a larger number enters the first stack before a smaller number does, the larger number must move to the second stack before the smaller number enters the first stack -- otherwise the decreasing condition for the first stack is violated.
\end{enumerate}
\end{observation}

We now give a classification of when a sorting word is a DI word.

\begin{lemma}~\label{sorting_req}
A sorting word $W$ of a permutation $\pi$ of length $n$ is a DI word if and only if 
for all $i\in[n]$, if no \DI{N}s appear in $W$ between $\DI{N}_{\pi_i}$ and $\DI{C}_{\pi_i}$, then $\#_\DI{E}(W^x) = \#_\DI{N}(W^x)$ where $x$ satisfies $W(x) = \DI{N}_{\pi_i}$; that is at step $x$, the first stack is empty.


Alternatively, we can say a sorting word $W$ is a DI word of some permutation  $\pi$ of length $n$ if and only if every entry in  $\pi$ that preceded the entries output by a sequence of $\DI{C}$s must have moved to the second stack (and possibly the output) prior to this exodus.
\end{lemma}

\begin{proof}
Suppose that $W$ is a DI word of a permutation $\pi$ and for some $i\in[n]$, no \DI{N}s appear in $W$ between $\DI{N}_{\pi_i}$ and $\DI{C}_{\pi_i}$, and $W(x) = \DI{N}_{\pi_i}$.
Assume that the first stack is nonempty after stage $x$, and let $\pi_j$ be a symbol in the first stack.
Then $\DI{E}_{\pi_j}\lew\DI{N}_{\pi_i}\lew\DI{N}_{\pi_j}$, and since no \DI{N}s appear between $\DI{N}_{\pi_i}$ and $\DI{C}_{\pi_i}$, it follows that $\DI{N}_{\pi_i} \lew \DI{C}_{\pi_i}\lew\DI{N}_{\pi_j}$.
By Observation~\ref{obs:wordprops}~(\ref{obs:wordpropsNE2}), we have that $\DI{N}_{\pi_i}\lew\DI{E}_{\pi_j}$, which gives a contradiction.
So at stage $x$, the first stack must be empty.

Conversely, suppose $W$ is a sorting word of a permutation $\pi$ of length $n$ which is not a DI word.
Recall that since $W$ represents the sorting of $\pi$ through two stacks in series, any entries in the second stack will obey the increasing condition at all times.  
Consequently, the movements corresponding to $W$ when applied to $\pi$ must cause a violation of the decreasing condition on the first stack.
In such a case, we have two entries $\pi_i$ and $\pi_j$ where $\pi_i > \pi_j$ and at some point $\pi_i$ is below $\pi_j$ in the first stack.  
That is, $\DI{E}_{\pi_i} \lew \DI{E}_{\pi_j} \lew \DI{N}_{\pi_j} \lew \DI{N}_{\pi_i}$.

Since the end result of the sorting is the identity permutation, we must have $\DI{C}_{\pi_j}\lew\DI{C}_{\pi_i}$.
Therefore $\DI{E}_{\pi_i} \lew \DI{E}_{\pi_j} \lew \DI{N}_{\pi_j} \lew \DI{C}_{\pi_j} \lew \DI{N}_{\pi_i} \lew \DI{C}_{\pi_i}$.
Let $\DI{N}_{\pi_k}$ be the last $\DI{N}$ in $W$ to appear before $\DI{C}_{\pi_j}$.
Since the second stack is increasing and $W$ is a sorting word, $\pi_k\leq\pi_j$ and so $\DI{N}_{\pi_k}\lew \DI{C}_{\pi_k}\leqw \DI{C}_{\pi_j}$.
By the selection of $k$, no \DI{N}s appear in $W$ between $\DI{N}_{\pi_k}$ and $\DI{C}_{\pi_k}$.
Observe that $\pi_i$ is in the first stack (and hence the first stack is nonempty) at the step corresponding to $\DI{N}_{\pi_k}$.
\end{proof}

To similarly characterize our ADI words, we first introduce a lemma showing how ``output-greedy" Algorithm~\ref{sorting_algorithm} is.

%
%

\begin{lemma}~\label{output_greedy}
Let $n\geq 1$ and $i\in[n]$ such that $\pi_i\neq n$.
Let $\pi$ be a DI-sortable permutation of length $n$ with ADI word $W$.
If $\DI{C}_{\pi_i}\not\precw\DI{C}_{\pi_i+1}$, then $\DI{C}_{\pi_i}\lew\DI{E}_{\pi_i+1}$.
That is, if two consecutive values do not exit in two consecutive stages, then the larger value does not enter the machine until after the smaller value exits.
\end{lemma}

\begin{proof}
Assume by way of contradiction that $\DI{C}_{\pi_i}\not\precw\DI{C}_{\pi_i+1}$ and $\DI{E}_{\pi_i+1}\lew \DI{C}_{\pi_i}$.

First suppose that $\DI{N}_{\pi_i+1} \lew \DI{C}_{\pi_i}$ and let $x$ satisfy $W(x) = \DI{C}_{\pi_i}$.
Then necessarily $\pi_i+1$ must be immediately below $\pi_i$ in the second stack after stage $x-1$.
So Step 1 of Algorithm~\ref{sorting_algorithm} applies at stage $x+1$, giving that $W(x+1)=\DI{C}_{\pi_i+1}$.
Hence $\DI{C}_{\pi_i}\precw \DI{C}_{\pi_i+1}$, a contradiction.  

Alternatively, suppose $\DI{E}_{\pi_i+1} \lew \DI{C}_{\pi_i} \lew\DI{N}_{\pi_i+1}$, and let $x$ satisfy $W(x) = \DI{E}_{\pi_i+1}$.
Then $\DI{N}_{\pi_i} \lew\DI{E}_{\pi_i+1}$
by Observation~\ref{obs:wordprops}~(\ref{obs:wordpropsNE2}).
Therefore $\DI{N}_{\pi_i} \lew\DI{E}_{\pi_i+1} \lew\DI{C}_{\pi_i}$, meaning that at stage $x$, $\pi_i+1$ is pushed into the first stack while some symbol in the second stack is smaller than $\pi_i+1$ (namely $\pi_i$).
Therefore Step 3 of Algorithm~\ref{sorting_algorithm} should not be applied at stage $x$, which is a contradiction.
\end{proof}

\begin{example}  See Figure~\ref{DI_example} and notice when entries $1,2,3$ are output, $4$ is still in the input.
This can also be seen in Example \ref{ex:adi}, where $\DI{C}_1\precw\DI{C}_2\precw\DI{C}_3\lew\DI{E}_4$.
\end{example}

We now conclude this section with a classification of when a DI word is the ADI word produced by Algorithm \ref{sorting_algorithm} for some DI-sortable permutation.

\begin{theorem}~\label{EN_before_C}
A DI word $W$ of a DI-sortable permutation $\pi$ of length $n$ is its ADI word if and only if $\DI{E}_1\precw\DI{N}_1\precw\DI{C}_1$ and for each $i\in[n]$ such that $\pi_i\neq 1$, either $\DI{C}_{\pi_i-1}\precw\DI{C}_{\pi_i}$ or $\DI{E}_{\pi_i}\precw\DI{N}_{\pi_i}\precw\DI{C}_{\pi_i}$.
That is, a DI word $W$ of a DI-sortable permutation $\pi$ is its ADI word if and only if any maximal sequence of consecutive copies of $\DI{C}$s in $W$ is immediately preceded by $\DI{EN}$.\end{theorem}

\begin{proof}   
Suppose $W$  is the ADI word of  $\pi$ and $i\in[n]$.
It follows by Algorithm \ref{sorting_algorithm} that $\DI{E}_1\precw\DI{N}_1\precw\DI{C}_1$, so we may assume going forward that $\pi_i\neq 1$.
Additionally any references to \emph{Steps} in the following argument refer to Algorithm \ref{sorting_algorithm}.

Assume that $\DI{C}_{\pi_{i}-1}\not\precw\DI{C}_{\pi_i}$.
By Lemma \ref{output_greedy}, $\DI{C}_{\pi_i-1}\lew\DI{E}_{\pi_i}$.
Define $x$ so that $W(x)=\DI{E}_{\pi_i}$.
Therefore after stage $x-1$, $\pi_i$ is smaller than all symbols in the second stack and larger than all symbols in the first stack -- otherwise Step 3 would not be applied at stage $x$.
In fact, all symbols smaller than $\pi_i$ have been output by stage $x-1$ since $\DI{C}_{\pi_i-1}\lew\DI{E}_{\pi_i}$.
Therefore after stage $x$, $\pi_i$ is the only symbol in the first stack.
Hence Step 2 applies at stage $x+1$, giving $W(x+1)=\DI{N}_{\pi_i}$.
Then Step 1 applies at stage $x+2$ giving that $W(x+2)=\DI{C}_{\pi_i}$ and therefore $\DI{E}_{\pi_i}\precw\DI{N}_{\pi_i}\precw\DI{C}_{\pi_i}$.

Conversely, suppose the DI word $W$ is not the ADI word of $\pi$, and let $A$ denote the ADI word of $\pi$.  
Let $t\in[3n]$ be the largest integer such that $A^{t-1}=W^{t-1}$, that is $t$ is the smallest integer so that  $A(t)\neq W(t)$.
We consider the state of the DI machine after stage $t-1$.
Define $i\in[n]$ so that $\DI{C}_{\pi_i-1}\leqw W(t-1)\lew\DI{C}_{\pi_i}$ (and also $\DI{C}_{\pi_i-1}\leqw[A] W(t-1)\lew[A]\DI{C}_{\pi_i}$).
By convention, we say that $\DI{C}_0$ occurs at position 0, and since $W(t)<\DI{C}_n$, the index $i$ is well-defined. 
Therefore all symbols smaller than $\pi_i$ have been output by stage $t-1$ and no \DI{C}s appear between $W(t-1)$ and $\DI{C}_{\pi_i}$ in $W$ and $A$.
Hence by stage $t-1$, $\pi_i$ is either at 
\begin{itemize}
\item the top of the second stack ($\DI{E}_{\pi_i}\lew[] \DI{N}_{\pi_i}\leqw[] W(t-1)\lew[]\DI{C}_{\pi_i}$ in $W$ and $A$), 
\item at the bottom of the first stack ($\DI{E}_{\pi_i}\leqw[] W(t-1)\lew[] \DI{N}_{\pi_i}\lew[]\DI{C}_{\pi_i}$ in $W$ and $A$), or 
\item part of the input yet to enter the stacks ($W(t-1)\lew[]\DI{E}_{\pi_i}\lew[] \DI{N}_{\pi_i}\lew[]\DI{C}_{\pi_i}$ in $W$ and $A$).
\end{itemize}
Suppose first that $\pi_i$ is located at the top of the second stack by stage $t-1$.
See Figure \ref{fig:ADIproof1}(a).
Observe that Step 1 of Algorithm \ref{sorting_algorithm}  outputs $\pi_i$ at stage $t$, so $A(t)=\DI{C}_{\pi_i}$.
Since $\pi_i$ is smaller than all symbols in the first stack (possibly vacuously), $W(t)\neq\DI{N}$.
Thus $W(t)=\DI{E}$.
Since $\pi_i$ is smaller than any symbol in the first stack and in the remaining input, no \DI{N} can appear in $W$ between $W(t)$ and $\DI{C}_{\pi_i}$; otherwise the second stack is no longer increasing.
Therefore $\DI{E}\precw\DI{C}_{\pi_i}$, and hence neither $\DI{C}_{\pi_i-1}\precw\DI{C}_{\pi_i}$ nor $\DI{N}_{\pi_i}\precw\DI{C}_{\pi_i}$.


\begin{figure}[ht]
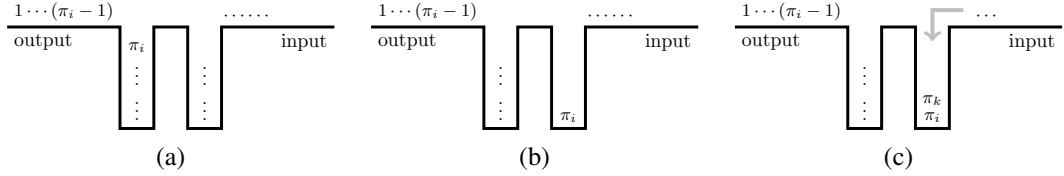

\centering
\begin{tabular}{ccc}
\pic[.75]{32} & \pic[.75]{33} & \pic[.75]{34} \\
(a) & (b) & (c) \\
\end{tabular}
\caption{Some states of the DI machine from the first two cases of the proof of Theorem~\ref{EN_before_C}. }
\label{fig:ADIproof1}
\end{figure}

Next suppose that $\pi_i$ is at the bottom of the first stack by stage $t-1$. 
Then $W(t-1)\lew[A]\DI{N}_{\pi_i}$ and hence $\DI{C}_{\pi_i-1}\leqw[A] W(t-1)\lew[A]\DI{N}_{\pi_i}\lew[A] \DI{C}_{\pi_i}$.
So $\DI{C}_{\pi_i-1}\not\precw[A]\DI{C}_{\pi_i}$, and hence $\DI{C}_{\pi_i-1}\lew[A]\DI{E}_{\pi_i}\leqw[A]W(t-1)$ by Lemma~\ref{output_greedy}.
Since $W^{t-1}=A^{t-1}$, we have that $\DI{C}_{\pi_i-1}\lew[W]\DI{E}_{\pi_i}$ as well.
Let $W(x)=\DI{E}_{\pi_i}$ (so $x\leq t-1$ and thus $A(x)=\DI{E}_{\pi_i}$).
Then Step 3 applies at stage $x$, meaning that $\pi_i$ is larger than all symbols in the first stack by stage $x-1$, but since $\DI{C}_{\pi_i-1}\lew[A]\DI{E}_{\pi_i}$ all symbols smaller than $\pi_i$ have been output.
So $\pi_i$ is alone in the first stack after stage $x$.
Since $\pi_i$ is smaller than any symbol (if any) in the second stack, it follows that $A(x)=\DI{E}_{\pi_i}\precw[A]\DI{N}_{\pi_i}\precw[A]\DI{C}_{\pi_i}$.
So $A(x+1)=\DI{N}_{\pi_i}$.
Since $A(t-1)\lew[A] \DI{N}_{\pi_i} = A(x+1)$, 
$t-1 < x+1$.
So $x=t-1$ and hence $A(t)=\DI{N}_{\pi_i}$ and $A(t-1)=W(t-1)=\DI{E}_{\pi_i}$.
Observe that $\pi_i$ is alone in the first stack by stage $t-1$ and 
$W(t)\neq\DI{C}$ (because $\DI{C}_{\pi_i-1}\leqw W(t-1) = \DI{E}_{\pi_i} \precw W(t)\lew\DI{C}_{\pi_i}$), 
so $W(t)=\DI{E}_{\pi_k}$ for some $\pi_k>\pi_i$.
Figures \ref{fig:ADIproof1} (b) and (c) show the states of the DI-machine using $W$ on $\pi$ after stages $t-1$ and $t$, respectively.
%

Since $\DI{C}_{\pi_i-1}\leqw W(t-1) = \DI{E}_{\pi_i}$, no \DI{C}s appear between $\DI{E}_{\pi_i}$ and $\DI{C}_{\pi_i}$ in $W$.
Hence $\DI{C}\not\precw\DI{N}_{\pi_i}$ and there are no \DI{C}s between $\DI{N}_{\pi_i}$ and $\DI{C}_{\pi_i}$.
Furthermore $\DI{E}\not\precw\DI{N}_{\pi_i}$ since $\DI{E}_{\pi_i}\not\precw\DI{N}_{\pi_i}$ and otherwise $W$ is not a sorting word.
Therefore $\DI{N}\precw \DI{N}_{\pi_i}$.
Furthermore, there are no \DI{N}s between $\DI{N}_{\pi_i}$ and $\DI{C}_{\pi_i}$, otherwise Observation~\ref{obs:wordprops}~(\ref{obs:wordpropsNNCCC}) is violated.
So either $\DI{N}\precw\DI{N}_{\pi_i}\precw\DI{C}_{\pi_i}$ or $\DI{E}\precw\DI{C}_{\pi_i}$.

Last, we show that $\pi_i$ must be in one of the stacks by stage $t-1$, which will conclude the proof.
Assume to the contrary that $\pi_i$ is not in either stack by stage $t-1$.
Clearly both $A(t)$ and $W(t)$ are not $\DI{C}$ (necessarily $\DI{C}_{\pi_i}$), so either
$A(t)=\DI{N}$ and $W(t)=\DI{E}$ or $A(t)=\DI{E}$ and $W(t)=\DI{N}$.
Suppose at stage $t-1$, $\pi_k$ is the next symbol to enter the stacks, and if the stacks are nonempty, $\pi_q$ sits atop the first stack and $\pi_m$ sits atop the second stack.
See Figure \ref{fig:ADIproof2}(a).
\begin{figure}[ht]
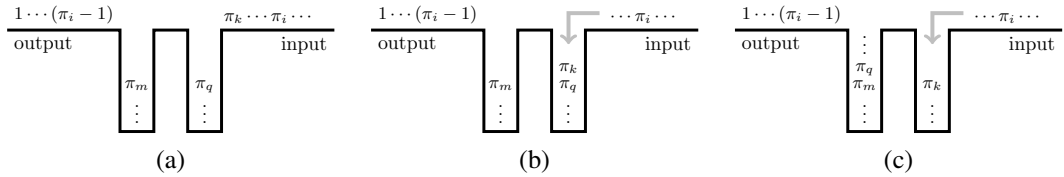

\centering
\vspace{-6pt}
\begin{tabular}{ccc}
\pic[.75]{35} & \pic[.75]{36} & \pic[.75]{37} \\
(a) & (b) & (c) \\
\end{tabular}
\caption{Some stages of the DI machine from the final cases of the proof of Theorem~\ref{EN_before_C}.  }
\label{fig:ADIproof2}
\end{figure}

First assume that $A(t)=\DI{N}_{\pi_q}$ and $W(t)=\DI{E}_{\pi_k}$ (and therefore the first stack is nonempty by stage $t-1$).
Then either Step 2 or Step 4 of Algorithm \ref{sorting_algorithm} applies at stage $t$.
However Step 2 cannot be satisfied because $\pi_i$ is the next symbol to be output, so Step 4 is applied, meaning the conditions of Step 3 are not satisfied. Hence either the second stack is nonempty and $\pi_k > \pi_m$, or $\pi_k < \pi_q$.
Since $W$ is a DI word and $W(t)=\DI{E}_{\pi_k}$, it follows that $\pi_k > \pi_q$ (and thus $\pi_k\neq \pi_i)$, which implies that the second stack must be nonempty and $\pi_k>\pi_m$; hence $\DI{N}_{\pi_m}\lew\DI{E}_{\pi_k}$.
Since $\DI{E}_{\pi_k}\lew\DI{E}_{\pi_i}$ we have that 
 $W(t) =\DI{E}_{\pi_k} \lew \DI{N}_{\pi_k} \lew\DI{E}_{\pi_i}$ by Observation~\ref{obs:wordprops}~(\ref{obs:wordpropsENE}).

So $\DI{C}_{\pi_i} \lew\DI{C}_{\pi_m}$ (since $\pi_i<\pi_m$), 
$\DI{C}_{\pi_m}\lew\DI{N}_{\pi_k}$ (follows from $\DI{N}_{\pi_m}\lew\DI{E}_{\pi_k}$, $\pi_m<\pi_k$, and Observation~\ref{obs:wordprops}~(\ref{obs:wordpropsCN})), and $\DI{N}_{\pi_k}\lew\DI{E}_{\pi_i}$ (because $\pi_i<\pi_k$, $\DI{E}_{\pi_k}\lew\DI{E}_{\pi_i}$, and Observation ~\ref{obs:wordprops}~(\ref{obs:wordpropsENE})).
Therefore $\DI{C}_{\pi_i}\lew \DI{E}_{\pi_i}$, a contradiction.
Hence we cannot have $A(t)=\DI{N}$ and $W(t)=\DI{E}$.

\newpage
Therefore $A(t)=\DI{E}_{\pi_k}$ and $W(t)=\DI{N}_{\pi_q}$.
This implies that $\pi_k > \pi_q$ (and thus $\pi_k\neq\pi_i$) and the first stack is nonempty at stage $t-1$, respectively.
This also gives that $W(t)=\DI{N}_{\pi_q}\lew\DI{E}_{\pi_k}$.
Furthermore $\pi_q>\pi_i$ since all other symbols smaller than $\pi_i$ have been output by stage $t-1$.

So $\DI{C}_{\pi_q}\lew\DI{N}_{\pi_k}$ because from $\DI{N}_{\pi_q}\lew\DI{E}_{\pi_k}$, $\pi_q<\pi_k$, and Observation~\ref{obs:wordprops}~(\ref{obs:wordpropsCN}). 
Furthermore $\DI{N}_{\pi_k}\lew\DI{E}_{\pi_i}$ since $\pi_i<\pi_k$, $\DI{E}_{\pi_k}\lew\DI{E}_{\pi_i}$, and Observation ~\ref{obs:wordprops}~(\ref{obs:wordpropsENE}). 
Since $\pi_i<\pi_q$, it follows that $\DI{C}_{\pi_i}\lew\DI{C}_{\pi_q}$.
So $\DI{C}_{\pi_q}\lew\DI{N}_{\pi_k}\lew\DI{E}_{\pi_i}\lew\DI{C}_{\pi_i}\lew\DI{C}_{\pi_q}$, which is a contradiction.
\end{proof}

We can now strengthen Theorem~\ref{EN_before_C}.
This result will be used in an argument in Section \ref{sec:properties}, specifically the proof of Theorem~\ref{thm:decomp}.

\begin{corollary}~\label{greedy}  
Consider the ADI word $W$ of a DI-sortable permutation $\pi$. 
%
Let $x\in[3n-1]$ and $\pi_i = \min\{\pi_k\mid W(x)\lew\DI{C}_{\pi_k}\}$.
If $\DI{E}_{\pi_i}\lew W(x)$, then $W(x)\neq \DI{E}$.
That is, Algorithm~\ref{sorting_algorithm} will output all available entries in the stacks before another push from the input is allowed.
\end{corollary}

\begin{proof}
If $W(x)=\DI{C}_{\pi_i-1}$, then clearly $W(x)\neq\DI{E}$.
If $\DI{C}_{\pi_i-1}\lew W(x)\lew \DI{C}_{\pi_i}$, then by Theorem~\ref{EN_before_C}, $\DI{E}_{\pi_i}\precw \DI{N}_{\pi_i}\precw \DI{C}_{\pi_i}$.
So $W(x)=\DI{N}_{\pi_i}\neq \DI{E}$.
\end{proof}

%
%

\section{The Bijection}
\label{sec:bijection}
In this section we give a bijection that takes DI-sortable permutations to their corresponding Schr\"oder paths via their ADI words.  
Let $\pi$ be a DI-sortable permutation of length $n$ with corresponding ADI word $W$.
To produce a Schr\"oder path, we use the following constructions.
Suppose that $W$ contains $k$ maximal, consecutive substrings consisting only of \DI{C}s.
Let $\tau=(\tau_1,\tau_2,\dots,\tau_k)$ be the partition of $n$ such that $\tau_i$ is the length of the $i$th such substring.
Define a $k$-tuple $\ell$ so that $\ell_1=\tau_1$ and $\ell_i = \tau_i+\ell_{i-1}$, 
when $2\leq i \leq k$.  Note that $\ell_k =n$ since it counts all the \DI{C}s of $W$.
Finally, define an increasing $k$-tuple $\rho$ so that for $x\in[3n]$ we have $\DI{E}\precw W(x)\precw\DI{C}$ if and only if $W(x)=\DI{N}$ and $\#_{\DI{N}}(W^x) = \rho_i$.
Note that $\rho$ is well-defined by Theorem \ref{EN_before_C}.
In other words, we have:
\begin{itemize}
\item $\ell_i$ is the number of \DI{C}s  in the first $i$ maximal subsequences of \DI{C}s in $W$.
\item $\rho_i$ is the number of \DI{N}s in $W$ before the $i\th$ maximal subsequence of \DI{C}s.
By Theorem \ref{EN_before_C}, the $\rho_i\th$ \DI{N} in $W$ is preceded by an \DI{E} and succeeded by a \DI{C}.
\end{itemize}



\begin{example}
\label{ex:rho_tau}
Let $\pi=81736245$.
Then $W = \DI{EN\,\textcolor{red}{ENC}\,ENEENN\,\textcolor{red}{ENCC}\,\textcolor{red}{ENC}\,\textcolor{red}{ENCCCC}}$, which contains $k=4$ maximal subsequences of consecutive \DI{C}s having lengths $1$, $2$, $1$, and $4$, respectively.
So  $\tau = (1,2,1,4)$ and hence $\ell = (1,3,4,8)$, and $\rho = (2,6,7,8)$.
The maximal substrings of the form $\DI{ENC}^{\tau_i}$ are given in red.
\end{example}

Using this, we give an algorithm for producing a Schr\"oder path from a DI-sortable permutation.	

\begin{algorithm}~\label{ADI_word_to_path}
Let $\pi$ be a DI-sortable permutation of length $n$.
\begin{enumerate}
\item
Use Algorithm~\ref{sorting_algorithm} to find the ADI word $W$ of $\pi$.  From $W$, obtain $\tau$ and $\ell$. 
\item 
Construct a word $T$ by replacing each maximal consecutive substring in $W$ of the form \DI{ENC}$^a$ with an~\DI{N}.
Observe that $\DI{N}$ is the last letter of $T$, and $\#_{\DI{N}}(T)=n$.
\item 
Construct a word $S_\DI{D}$ from $T$ by replacing the $\ell_i$\textsuperscript{th} \DI{N} in $T$ with a \DI{D} for each $i\in[k]$.
Observe that since $\ell_k=n$, the final character in this word is \DI{D}.  
Remove this \DI{D} to produce $S$.
\end{enumerate}
\end{algorithm}

\begin{example}  
\label{ex:reverse}
Let $\pi = 81736245$ as given in Example \ref{ex:rho_tau}, in which we give $W$, $\tau$, and $\ell$.
We now apply Algorithm~\ref{ADI_word_to_path} to $\pi$.
We first produce $T= \DI{E\textcolor{blue}{N}NE\textcolor{blue}{N}EE\textcolor{blue}{N}NNN\textcolor{blue}{N}}$ by replacing the maximal consecutive substrings \DI{ENC}$^{\tau_i}$ ($1\leq i \leq 4$) with an \DI{N}.

Next, we replace the $\ell_i$\textsuperscript{th} \DI{N}s in $T$ with \DI{D} (in blue) to produce $S_\DI{D}=\DI{EDNEDEEDNNND}$, then remove the final \DI{D} to produce $S = \DI{EDNEDEEDNNN}$.
See Figure~\ref{81736245}.
\end{example}

\begin{figure}[ht]
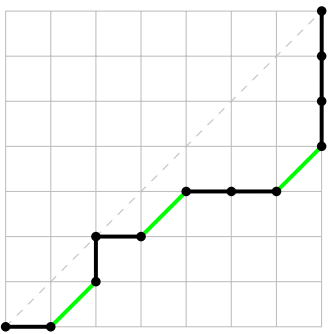

\centering
\pic{1}
\caption{The Schr\"oder path \DI{EDNEDEEDNNN} corresponding to $\pi=81736245$.}
\label{81736245}
\end{figure}

We now give an observation, then prove two lemmas which collectively show that Algorithm \ref{ADI_word_to_path} provides a one-to-one correspondence between ADI words and Schr\"oder paths.

\begin{observation}
\label{obs:rho_tau}
Let $W$ be an ADI word for a permutation $\pi$ of length $n$ with $k$ substrings of the form $\DI{ENC}^{\tau_i}$ and $k$-tuples $\rho$ and $\ell$ as defined earlier, and let $\ell_0=0$.
Observe that $\DI{C}_m$ is the $m$th occurrence of \DI{C} in $W$.
Let $i\in[k]$ and  $\hat{x},x \in[3n]$ such that  $W(\hat{x})\precw\DI{C}_{\ell_{i-1}+1}$ and $W(x)=\DI{C}_{\ell_i}$; in other words $W(\hat{x})$ is the \DI{N} which immediately precedes the $i$th maximal subsequence of \DI{C}s and $W(x)$ is the last \DI{C} in the $i$th maximal substring of \DI{C}s.
See Figure \ref{fig:rho_tau}.
Then, $\ell_i=\#_\DI{C}(W^x) \leq \#_\DI{N}(W^x) = \#_\DI{N}(W^{\hat{x}}) = \rho_i$.
\end{observation}
Observe that in Example \ref{ex:rho_tau}, $\ell_i \leq \rho_i$ for each $i\in[4]$.
\begin{figure}[h]
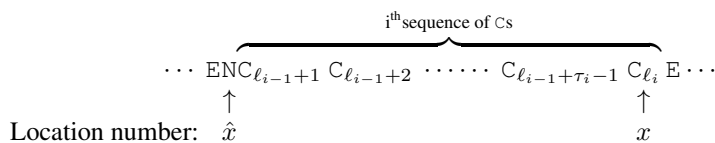

\centering
\begin{tabular}[t]{r@{}}
$\cdots$\\
\phantom{$\uparrow$}\\
Location number:
\end{tabular}
\DI{E}\begin{tabular}[t]{@{}c@{}}\DI{N}%
\\$\uparrow$\\$\hat{x}$\end{tabular}$
\overbrace{\DI{C}_{\ell_{i-1}+1}\ \DI{C}_{\ell_{i-1}+2}\ \cdots\cdots\ \DI{C}_{\ell_{i-1}+\tau_i-1}\ \begin{array}[t]{@{}c@{}}\DI{C}_{\ell_i}\\\uparrow\\x\end{array}}^{\text{i\th sequence of $\DI{C}$s}}\DI{E}\cdots$
\caption{An illustration of the indexing used for a word $W$ in Observation \ref{obs:rho_tau}.}
\label{fig:rho_tau}
\end{figure}

\begin{lemma}~\label{forward}
The ADI word for a DI-sortable permutation of length $n>0$ yields a Schr\"oder path from $(0,0)$ to $(n-1,n-1)$ via Algorithm~\ref{ADI_word_to_path}.
\end{lemma}

\begin{proof}  
Let $W$ be the ADI word for a DI-sortable permutation $\pi$ of length $n$.
Suppose that $W$ contains $k$ disjoint maximal consecutive substrings of \DI{C}s.
Let $T$, $S_\DI{D}$, and $S$ be the words produced in Algorithm~\ref{ADI_word_to_path}.
Then the length of $T$ is $2n-k$ with $\#_\DI{N}(T) = n$ and $\#_\DI{E}(T)=n-k$.
Therefore the length of $S$ is $2n-k-1$ with $\#_\DI{N}(S) = \#_\DI{E}(S)=n-k$ and $\#_\DI{D}(S)=k-1$.
It follows that $S$ corresponds to a path from $(0,0)$ to $(n-1,n-1)$.

Let $x\in[2n-k-1]$ and suppose that $S(x)=\DI{N}$.
It is sufficient to show that $\#_\DI{E}(S^x) \geq \#_\DI{N}(S^x)$.
We have $T(x)=\DI{N}$.
Let $i\in[n]$ such that $\rho_{i-1} < \#_\DI{N}(T^x) \leq \rho_i$, with the convention that $\ell_0=0, \rho_0=0$.
Since $\ell_{i-1}\leq \rho_{i-1}$ and 
 in producing $S$ we replace the $\ell_j$\textsuperscript{th} \DI{N} in $T$ with a \DI{D}
whenever $1 \leq j < k$, we have $\#_\DI{D}(S^x) \geq i-1$.  
Let $\hat{x}\in[3n]$ so that $W(\hat{x})=\DI{N}$ and corresponds to the \DI{N} at $T(x)$.
Then $\#_\DI{N}(T^x)=\#_\DI{N}(W^{\hat{x}})$.
Furthermore $\DI{C}_{\ell_{i-1}}\lew W(\hat{x}) \lew \DI{C}_{\ell_{i-1}+1}$, and since, in producing $T$, we remove $i-1$ instances of $\DI{E}$ that precede the $\DI{N}$ at location $\hat{x}$ in $W$, it follows that $\#_\DI{E}(W^{\hat{x}}) = \#_\DI{E}(T^x)+(i-1)$.
In addition, observe that $\#_\DI{E}(S^x) = \#_\DI{E}(T^x)$ and $\#_\DI{N}(T^x) = \#_\DI{N}(S^x)+\#_\DI{D}(S^x)$.
So  
\[\begin{array}{rcl}
\#_\DI{E}(S^x) &=& \#_\DI{E}(T^x) = \#_\DI{E}(W^{\hat{x}})-(i-1) \\
&\geq& \#_\DI{N}(W^{\hat{x}})-(i-1) = \#_\DI{N}(T^x) - (i-1) = \#_\DI{N}(S^x)+\#_\DI{D}(S^x)-(i-1) \geq \#_\DI{N}(S^x).
\end{array}\]
Hence $S$ is a Schr\"oder path from $(0,0)$ to $(n-1,n-1)$.
\end{proof}

\begin{lemma}~\label{back}
For every Schr\"oder path $S$ from $(0, 0)$ to $(n -1, n-1)$, there is an ADI word $W$ of a DI-sortable permutation of length $n$ such that Algorithm~\ref{ADI_word_to_path} applied to $W$ gives $S$.
\end{lemma}

\begin{proof}  We begin with a Schr\"oder path from $(0,0)$ to $(n-1,n-1)$ and consider it as a word $S$.  
Let $k=\#_\DI{D}(S) +1$.
Then $\#_\DI{E}(S) = \#_\DI{N}(S) = n-k$.
Furthermore, we have $\#_\DI{E}(S') \geq \#_\DI{N}(S')$ for any prefix $S'$ of $S$, since the Schr\"oder path stays weakly below the main diagonal.
Following Algorithm~\ref{ADI_word_to_path} backwards, append a \DI{D} to the end of $S$ to create the ``extended'' Schr\"oder path $S_{\DI{D}}$.  

For each $i\in[k]$, let $t_i$ be the location of the $i\th$ \DI{D} in $S_\DI{D}$, and let $\ell_i$ be the location of the $i\th$ \DI{D} in the substring of $S_{\DI{D}}$ consisting of all of the \DI{N}s and \DI{D}s of $S_\DI{D}$.
Since the substring of $S_{\DI{D}}$ consisting of all of the \DI{N}s and \DI{D}s of $S_\DI{D}$ has exactly $n$ letters, observe that $\ell_k=n$ and for convention, we set $\ell_0=0$.

Next create $T$ by replacing each \DI{D} with an \DI{N} in $S_{\DI{D}}$.  
Then $\#_\DI{N}(T) = n = k + \#_\DI{E}(T)$, and
observe that if $t_i \leq x < t_{i+1}$ for some $i\in[k-1]$, then 
$\#_\DI{D}(S_\DI{D}^x) = i$ and  
$\#_\DI{N}(S_\DI{D}^x) \leq \#_\DI{E}(S_\DI{D}^x) = \#_\DI{E}(T^x)$. 
So 
\[\#_\DI{N}(T^x) =\#_\DI{N}(S_\DI{D}^x) + \#_\DI{D}(S_\DI{D}^x) \leq \#_\DI{E}(T^x)+i.\]

Define $r_i=\min\{ x\mid \#_\DI{N}(T^x) = \#_\DI{E}(T^x)+i\}$ for each $i\in[k]$.
In other words, $r_i$ is the minimal length for a prefix of $T$ which has exactly $i$ more \DI{N}s than \DI{E}s.
By convention, let $r_0=0$.

Observe that $T(r_i) = \DI{N}$ and $\#_\DI{N}(T^{r_i-1}) = \#_\DI{E}(T^{r_i-1}) + i-1$ for each $i\in[k]$ and that $r_k \leq 2n-k$ since $\#_\DI{N} (T^{2n-k}) = \#_\DI{E} (T^{2n-k}) +k$.
Furthermore, note that $r_i \geq t_i$ for each $i\in[k]$.
Indeed, 
\[ \#_\DI{D}(S^{r_i}_\DI{D}) = 
\#_\DI{N}(T^{r_i}) - \#_\DI{N}(S^{r_i}_\DI{D}) = 
\#_\DI{E}(T^{r_i}) + i - \#_\DI{N}(S^{r_i}_\DI{D}) = 
\#_\DI{E}(S^{r_i}_\DI{D}) + i - \#_\DI{N}(S^{r_i}_\DI{D}) \geq i,\]
since $S_{\DI{D}}$ is a Schr\"oder path from $(0,0)$ to $(n,n)$.
In particular, $r_k \geq t_k = 2n-k$ hence $r_k = 2n-k$.

Now create $\hat{T}$ by inserting an \DI{E} before the $k$ \DI{N}s at locations $r_1,r_2,\dots,r_k$ of $T$.
Let $\hat{r}_i$ be the new location of these \DI{N}s in $\hat{T}$, that is $\hat{r}_i = r_i + i$, meaning the inserted \DI{E}s are at locations $\hat{r}_i-1$ in $\hat{T}$.
Again by convention, let $\hat{r}_0=0$.
We have $\#_\DI{E}(\hat{T}) = n = \#_\DI{N}(\hat{T})$ and $\hat{r}_k = 2n$.

We now establish that $\#_\DI{N}(\hat{T}^x)\leq \#_\DI{E}(\hat{T}^x)$ for each $x\in[2n]$.
First suppose $x=\hat{r}_i-1$ for some $i\in[k]$.
Then $x-i = r_i-1$.
Furthermore $\#_\DI{N}(\hat{T}^x)=\#_\DI{N}(T^{x-i})=\#_\DI{N}(T^{r_i-1})=\#_\DI{E}(T^{r_i-1})+i-1$ and 
\[\#_\DI{E}(\hat{T}^x)=\#_\DI{E}(T^{r_i-1})+i=\#_\DI{N}(T^{r_i-1})+1 = \#_\DI{N}(\hat{T}^x)+1.\]
Observe that since $\hat{T}(\hat{r}_i)=\DI{N}$, we have $\#_\DI{E}(\hat{T}^{\hat{r}_i})=\#_\DI{N}(\hat{T}^{\hat{r}_i})$.
This will be used later.

Now suppose $\hat{r}_i \leq x < \hat{r}_{i+1}-1$ where $0\leq i\leq k-1$.
Then $r_i \leq x-i < r_{i+1}$, so $\#_\DI{N}(\hat{T}^x) = \#_\DI{N}(T^{x-i})$.
So by definition of $r_{i+1}$,
\[
\#_\DI{N}(\hat{T}^x)
=\#_\DI{N}(T^{x-i})
\leq
\#_\DI{E}(T^{x-i})+i
=\#_\DI{E}(\hat{T}^x).\]
Additionally since $\hat{r}_k = 2n$, we have $\#_\DI{N}(\hat{T}^{\hat{r}_k})=n=\#_\DI{E}(\hat{T}^{\hat{r}_k})$.
Thus $\#_\DI{N}(\hat{T}^{x})\leq\#_\DI{E}(\hat{T}^{x})$ for all $x\in[2n]$.

Now, create the partition $\tau=(\tau_1,\tau_2,\dots,\tau_k)$ of $n$ where $\tau_i = \ell_i-\ell_{i-1}$ for each $i\in[k]$.
Next, create $W$ by inserting $\tau_i$ copies of \DI{C} after the \DI{N}s in $\hat{T}$ located at $\hat{r}_i$ for each $i\in[k]$.
Let $\doublehat{r}_i$ be the new location of these \DI{N}s in $W$, that is $\doublehat{r}_i = \hat{r}_i + \ell_{i-1}$, meaning the inserted \DI{C}s are at locations $\doublehat{r}_i+1,\dots,\doublehat{r}_i+\tau_i$.

By its creation, it follows from the properties of $\hat{T}$ that $\#_\DI{E}(W)=\#_\DI{N}(W)=\#_\DI{C}(W)=n$ and $\#_\DI{N}(W^x) \leq \#_\DI{E}(W^x)$ for all $x\in[3n]$.
To show that $\#_\DI{C}(W^x) \leq \#_\DI{N}(W^x)$ for all $x\in[3n]$, we need only to show the inequality holds when $x=\doublehat{r}_i+\tau_i$ for each $i\in[k]$.
So let $x=\doublehat{r}_i+\tau_i$ for some $i\in[k]$.  Recall $t_i$ was defined to be the location of the $i\th$ \DI{D} in $S_\DI{D}$, and it was established that $t_i \leq r_i$ for all $i \in [k]$.
Then
\[\#_\DI{N}(W^x) 
= \#_\DI{N}(\hat{T}^{\hat{r}_i})
= \#_\DI{N}(T^{r_i})
\geq \#_\DI{N}(T^{t_i})
= \#_\DI{N}(S^{t_i}) + \#_\DI{D}(S^{t_i})
=\ell_i = \#_\DI{C}(W^x).
\]
Therefore $\#_\DI{C}(W^x) \leq \#_\DI{N}(W^x) \leq \#_\DI{E}(W^x)$ for each $x\in[3n]$.
So by Observation~\ref{observation}, $W$ is a sorting word for a permutation $\pi$ of length $n$.
Moreover, $\#_\DI{E}(W^{\doubleExphat{r}_i})=\#_\DI{N}(W^{\doubleExphat{r}_i})$ for each $i\in[k]$ since $\#_\DI{E}(\hat{T}^{\hat{r}_i})=\#_\DI{N}(\hat{T}^{\hat{r}_i})$.
From Lemma \ref{sorting_req}, $W$ is a DI word of $\pi$.

Finally, one can see by the construction of $W$ that $\DI{E}_{\ell_{i-1}+1}\precw\DI{N}_{\ell_{i-1}+1}\precw\DI{C}_{\ell_{i-1}+1}$ for each $i\in[k]$ and $\DI{C}_{\pi_i-1}\precw\DI{C}_{\pi_i}$ if $\pi_i\notin\{\ell_{i-1}+1\mid i\in[k]\}$.
Hence from 
Theorem~\ref{EN_before_C}, $W$ is an ADI word for some DI-sortable permutation $\pi$.  
Moreover, applying Algorithm~\ref{ADI_word_to_path} to $W$ gives $S$.
\end{proof}

By Lemmas \ref{forward} and \ref{back}, our main result follows: 

\begin{theorem}
For any $n>0$, Algorithm~\ref{ADI_word_to_path} produces a bijection between the set of ADI words from DI-sortable permutations of length $n$ and Schr\"oder paths from $(0,0)$ to $(n-1,n-1)$. 
\end{theorem}

\begin{example}
Suppose we start with the Schr\"oder path $S=\DI{EDNEDEEDNNN}$ illustrated in Figure~\ref{81736245}.  
We demonstrate the inverse of Algorithm~\ref{ADI_word_to_path} described in Lemma~\ref{back}.

\begin{figure}[ht]
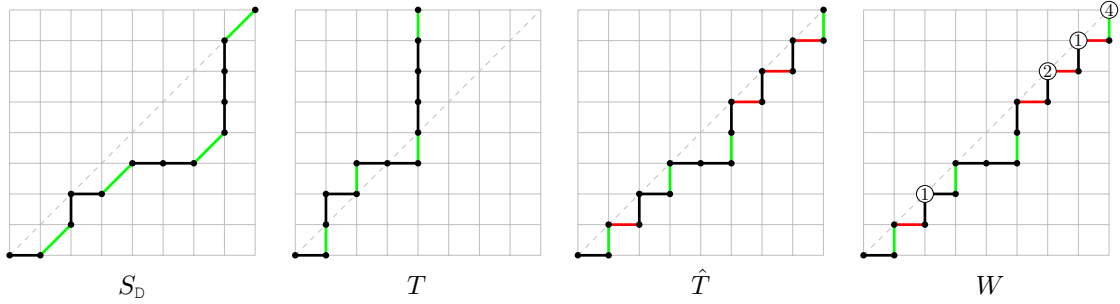

\centering
\begin{tabular}{@{}cccc@{}}
\pic[.68]{38} &
\pic[.68]{2} &
\pic[.68]{6} &
\pic[.68]{7} \\
$S_\DI{D}$ & $T$ & $\hat{T}$ & $W$
\end{tabular}
\caption{The process by which the Schr\"oder path in Figure \ref{81736245} becomes an ADI-word.
Circled numbers indicate consecutive copies of \DI{C}.}
\label{fig:graph_inverse}
\end{figure}

We begin by producing $S_\DI{D}=\DI{EDNEDEEDNNND}$ by appending a diagonal (Figure \ref{fig:graph_inverse}(a)).
We identify $\ell=(1,3,4,8)$ and $\tau=(1,2,1,4)$ based on the relative locations of the diagonal edges with respect to the north edges.
We construct $T=\DI{E\textcolor{blue}{N}NE\textcolor{blue}{N}EE\textcolor{blue}{N}NNN\textcolor{blue}{N}}$ by converting each diagonal edge to a north edge (Figure \ref{fig:graph_inverse}(b)).

Next we find $r=(3,10,11,12)$ and construct $\hat{T}=\DI{EN\textcolor{blue}{EN}ENEENN\textcolor{blue}{EN}\textcolor{blue}{EN}\textcolor{blue}{EN}}$ by inserting \DI{E} before the \DI{N} at each of the locations in $T$ specified by $r$.
Note that $\hat{r}=(4,12,14,16)$, and with this we construct $W=\DI{ENE\textcolor{blue}{NC}ENEENNE\textcolor{blue}{NCC}E\textcolor{blue}{NC}E\textcolor{blue}{NCCCC}}$.
Observe that $W$ is the ADI word of 81736245, and Algorithm~\ref{ADI_word_to_path} applied to $W$ gives $S$ (see Example \ref{ex:reverse}).
\end{example}

\section{Permutation properties observable from their Schr\"oder paths}
\label{sec:properties}

While the bijection obtained from Algorithm~\ref{ADI_word_to_path} can be used to show exactly which Schr\"oder path corresponds to a given DI-sortable permutation, there are several properties that translate prominently between permutation and path.  

\begin{definition}  A \emph{right-to-left minimum} of a permutation is an element that is the least element seen thus far when reading the permutation from right to left.
\end{definition}

\begin{example}  The permutation $\pi = 81736245$ has four right-to-left minima, namely $1,2,4,5$.
\end{example}

\begin{lemma}~\label{number_RtLm}
If $W$ is the ADI word of a DI-sortable permutation $\pi$, then $\DI{E}_{\pi_i}\precw\DI{N}_{\pi_i}\precw\DI{C}_{\pi_i}$ if and only if $\pi_i$ is a right-to-left minimum.
Furthermore, the number of right-to-left minima of a DI-sortable permutation $\pi$ is one more than the number of diagonal steps in its corresponding Schr\"oder path.
\end{lemma}

\begin{proof}
%
Suppose $\pi_i$ is a right-to-left minimum of $\pi$ with ADI word $W$.
There must be some $k\in[n]$ that satisfies $\DI{E}_{\pi_k}\precw\DI{N}_{\pi_k}\precw\DI{C}_{\pi_k}\precw\cdots\precw\DI{C}_{\pi_i}$ by Theorem~\ref{EN_before_C}.  Assume by way of contradiction that $k \neq i$.
Then $\DI{E}_{\pi_i}\lew\DI{E}_{\pi_k}$ and $\DI{C}_{\pi_k}\lew\DI{C}_{\pi_i}$.
Hence $i < k$ and $\pi_k < \pi_i$, which is a contradiction to $\pi_i$ being a right-to-left minimum.
So $\DI{E}_{\pi_i}\precw\DI{N}_{\pi_i}\precw\DI{C}_{\pi_i}$.


Now suppose that $\DI{E}_{\pi_i}\precw\DI{N}_{\pi_i}\precw\DI{C}_{\pi_i}$ for some $i\in[n]$ and let $j\in[n]$ such that $i < j$.
Then $\DI{E}_{\pi_i} \lew \DI{E}_{\pi_j}$.
Therefore $\DI{C}_{\pi_i} \lew \DI{E}_{\pi_j}$, so $\DI{C}_{\pi_i}\lew \DI{C}_{\pi_j}$.
Hence $\pi_i < \pi_j$.
So $\pi_i$ is a right-to-left minimum.
%
%
%
%

Hence there is a one-to-one correspondence between consecutive sequences \DI{ENC} in $W$ and right-to-left minima in $\pi$. And as seen in Algorithm \ref{ADI_word_to_path}, these sequences give rise to one fewer diagonal in the corresponding Schr\"oder path of $\pi$.
\end{proof}

%
%

Lemma~\ref{number_RtLm} leads to a visual proof to the Corollary~\ref{Catalan}.

\begin{corollary}~\label{Catalan}
The DI-sortable permutations with one right-to-left minimum, that is the DI-sortable permutations that end in $1$, are enumerated by the Catalan numbers.  In particular, the number of DI-sortable permutations of length $n+1$ ending in $1$ is $\displaystyle{C_{n}=\frac{1}{n+1}{\binom{2n}{n}}}$.
\end{corollary}

\begin{proof}
The Schr\"oder paths from $(0,0)$ to $(n,n)$ without diagonal steps are Dyck paths, known to be enumerated by the Catalan numbers.

Alternatively, recall that the basis for our DI-sortable permutations is $\{3142, 3241\}$.  Thus, the DI-sortable permutations of length $n+1$ ending in $1$ are exactly the permutations whose first $n$ entries must avoid $213$, which are known to be enumerated by Catalan numbers, $C_n$~\cite{knuth:the-art-of-comp:1}.
\end{proof}

For a given diagonal in a Schr\"oder path $S$ and prefix $S'$ of $S$ terminating after the given diagonal, the we say the \emph{height} of the diagonal is $\#_\DI{N}(S')+\#_\DI{D}(S')$.
In other words, its height is the $y$-coordinate of the point in the path where the diagonal terminates.

\begin{example}
Let $\pi=81736245$ with corresponding Schr\"oder path $S$ given in Figure \ref{81736245}.
Then the height of its diagonals are $1$, $3$, and $4$.
\end{example}

\begin{theorem}~\label{diagonal_height}
The heights of the diagonals determine the values of the right-to-left minima.  
In particular, if the heights of the diagonals in the Schr\"oder path corresponding to the permutation $\pi$ are $h_1, h_2, \ldots, h_k$, the right-to-left minima of $\pi$ are $1, h_1+1, h_2+1,\ldots, h_k+1$.
\end{theorem}

\begin{proof}
Suppose that $S$ has $k$ diagonals (and hence $\pi$ has $k+1$ right-to-left minima), and let $W$ be the ADI word of $\pi$.
Notice since $\ell_i$ counts the \DI{C}s in the first $i$ maximal substrings of \DI{C}s, we have $\ell_i$ counts $\DI{C}_1,\DI{C}_2,\ldots, \DI{C}_{\ell_i}$.  
Thus $\DI{E}_{a}\precw\DI{N}_{a}\precw\DI{C}_{a}$ if and only if $a=\ell_i+1$ for some $i$ where $0\leq i \leq k$.
(Recall $\ell_0=0$ by convention.)
From Lemma \ref{number_RtLm} the right-to-left minima	are exactly $\ell_i+1$ for $0\leq i \leq k$.  

Let $x_1,x_2,\dots,x_k$ be locations in $S$ so that $x_i$ is the location of the $i\th$ \DI{D}; that is $S(x_i)=\DI{D}$ and $\#_\DI{D}(S^{x_i})=i$.
From Step 3 in Algorithm~\ref{ADI_word_to_path}, we observe that $\#_\DI{N}(S^{x_i})+\#_\DI{D}(S^{x_i})=\ell_i$ is the height of the $i\th$ diagonal for each $i\in[k]$.
So the heights of the diagonals of $S$ are one less than the right-to-left minima of $\pi$ (excluding 1).
\end{proof}

\begin{definition}  An \emph{interval} of a permutation $\pi$ is a consecutive subsequence of $\pi$ that contains consecutive values.
\end{definition}

\begin{example}  The permutation $\sigma = 685712943$ contains intervals $6857, 12, 9, 43$.
\end{example}

\begin{definition}
A permutation $\pi$ is said to be \emph{plus-decomposable} if $\pi$ is the concatenation of two non-empty intervals $\omega$ and $\tau'$ where the values of $\omega$ are less than those of $\tau'$.  Further, if we rescale the entries of $\tau'$ by subtracting the length of $\omega$ from each entry of $\tau'$ to get a permutation $\tau$, we denote $\pi = \omega \oplus \tau$.  If a permutation is not plus-decomposable, we say the permutation is \emph{plus-indecomposable}.
\end{definition}

\begin{example}  The permutation $\pi=43126758$ is plus-decomposable and can be written as $\pi = 4312 \oplus 231\oplus 1$.  The permutation $\sigma = 685712943$ is plus-indecomposable.
\end{example}

\begin{theorem}  
\label{thm:decomp}
The Schr\"oder path corresponding to a plus-decomposable DI-sortable permutation $\pi = \omega \oplus \tau$ is the Schr\"oder path for $\omega$ followed by a diagonal step (on the main diagonal) followed by the Schr\"oder path for $\tau$.  Furthermore, this means the Schr\"oder paths with diagonal steps on the main diagonal correspond exactly to the plus-decomposable DI-sortable permutations.
\end{theorem}

\begin{proof}
Consider the plus-decomposable permutation $\pi = \omega \oplus \tau$ and let $W$ be the ADI word of $\pi$.  
%
%
Let $x\in[3n-1]$ satisfy $W(x)=\DI{E}_{\pi_k}$ with $k=|\omega|+1$ and $\pi_i = \min\{\pi_\ell\mid W(x)\lew C_{\pi_\ell}\}$.
Then by Corollary~\ref{greedy}, we have $\DI{E}_{\pi_k} = W(x)\leqw\DI{E}_{\pi_i}$, so $k\leq i$.
Since $\pi = \omega\oplus\tau$ and $i\geq k = |\omega|+1$, we have $\pi_i>|\omega|$.
Hence by the definition of $\pi_i$, we have $\DI{C}_{|\omega|}\lew W(x) = \DI{E}_{\pi_k}$.
Thus we know Algorithm~\ref{sorting_algorithm} will force the exit of all the entries of $\omega$ before any entry of $\tau'$ (entries of $\pi$ corresponding to $\tau$) enters the stacks.  
Therefore the ADI word for $\pi$ will be the ADI word for $\omega$ followed by the ADI word for $\tau$.

When Algorithm~\ref{ADI_word_to_path} acts on the ADI word for $\pi$, it will convert the first portion of the ADI word to the Schr\"oder path for $\omega$, but followed by $\DI{D}$ as this $\DI{D}$ is no longer the last $\DI{D}$ of the word.  At the same time, it converts the remaining portion of the ADI word to the Schr\"oder path for $\tau$.
\end{proof}

\begin{example}  
\label{ex:plusdecomp}
In Figure~\ref{fig:plusdecomp}, we see that this Schr\"oder path for $\pi=43126758=4312\oplus231\oplus1$ is the concatenation of the Schr\"oder paths for $4312$, $231$, and $1$ (a degenerate path), connected by diagonal edges.
\end{example}

\begin{figure}[ht]
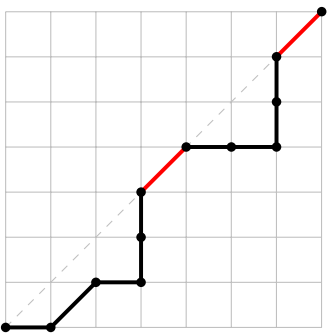

\centering
\pic{30}
\caption{The plus-decomposition of the Schr\"oder path of $43126758$ as given in Example \ref{ex:plusdecomp}.}
\label{fig:plusdecomp}
\end{figure}

Aguiar and Moreira~\cite{aguiar:baxter} showed that the number of Schr\"oder paths from $(0,0)$ to $(n,n)$ that do not have diagonal steps on the main diagonal equals the number of Schr\"oder paths from $(0,0)$ to $(n,n)$ that do have diagonal steps on the main diagonal as a particular case of a larger result.  That is, they showed the number of Schr\"oder paths that do not have diagonal steps on the main diagonal are enumerated by the small Schr\"oder numbers, which gives us the following corollary.

\begin{corollary}  The plus-decomposable DI-sortable permutations are in bijection with the plus-indecomposable DI-sortable permutations.  In particular, both classes are enumerated by the small Schr\"oder numbers.
\end{corollary}

We conclude with a few open questions:
\begin{enumerate}
\item Are there any other properties of the DI-sortable permutations that are easily seen by looking at the corresponding Schr\"oder paths?
\item  Is there another bijection between the DI-sortable permutations and the Schr\"oder paths that conserves more or different properties of the DI-sortable permutations in the Schr\"oder paths?
 \end{enumerate}
 
\acknowledgements
\label{sec:ack}
We would like to thank the referees for their careful reading of this paper as well as their multiple suggestions for improvement and clarity.

\bibliographystyle{acm}
\bibliography{schroder}
\label{sec:biblio}

\end{document}